\documentclass[10pt]{article}
\usepackage[utf8]{inputenc}
 
\usepackage[ruled,linesnumbered]{algorithm2e}

\usepackage{amssymb}
\usepackage{amsmath}
\usepackage{amsthm}
\usepackage{url}
\usepackage{mathrsfs}
\usepackage{pdfpages}
\usepackage{epigraph}
\usepackage{tikz-cd}

\newtheorem{lem}{Lemma}
\newtheorem{deff}{Definition}
\newtheorem{mytheor}{Theorem}

\newtheorem{theor}{Theorem}
\newtheorem{cor}{Corollary}
\newtheorem{rem}{Remark}

\newcommand{\meas}{\mathcal{M}}

\newcommand{\clNF}{\mathcal{NF}}

\newcommand{\N}{\mathbb{N}}

\newcommand{\C}{\mathbb{C}}

\newcommand{\sA}{\mathscr{A}}

\newcommand{\bB}{\mathcal{B}}
\newcommand{\bS}{\mathcal{S}}

\newcommand{\fa}{\mathfrak{a}}

\newcommand{\expect}{\mathbb{E}}

\DeclareMathOperator{\supp}{supp}

\DeclareMathOperator{\pr}{pr}

\DeclareMathOperator{\Span}{Span}

\begin{document}

\author{Andrei Alpeev  \footnote{DMA, École normale supérieure, Université PSL, CNRS, 75005 Paris, France,\\alpeevandrey@gmail.com}}
\title{A characterization of $C^*$-simplicity of countable groups via Poisson boundaries}

\maketitle

\begin{abstract}
We characterize $C^*$-simplicity for countable groups by means of the following dichotomy. If a group is $C^*$-simple, then the action on the Poisson boundary is essentially free for a generic measure on the group. If a group is not $C^*$-simple, then the action on the Poisson boundary is not essentially free for a generic measure on the group.
\end{abstract}

\section{Introduction}

Let $G$ be a discrete group. Consider its group algebra $\C[G]$. There is a natural linear representation of $\C[G]$ on $\C[G]$. This representation could be extended into a unitary representation of $\C[G]$ on $l^2(G)$ (since $\C[G]$ embeds densely into $l^2(G)$), The norm-closure $C_r(G)$ of the image of $C[G]$ in $\bB(l^2(G))$ is called the {\em reduced $C^*$-algebra of the group $G$}. A group is called $C^*$-simple if its reduced $C^*$-algebra is a simple $C^*$-algebra (meaning it does not have any non-trivial closed two-sided ideals). The question of $C^*$-simplicity for groups has drawn a significant attention in the previous decade when it was establsihed by Kalantar and Kennedy that it is equivalent to the topological freeness of the action of the group on its Furstenberg boundary (the maximal minimal strongly proximal action; strong proximality means that the orbit of any Borel probability measure on the space has a delta-measure in its weak-* closure). 
This approach was used later in \cite{BKKO17} to characterize the unique trace property as being equivalent to the triviality of the amenable radical (the maximal normal amenable subroup) of the group. It was conjectured that non $C^*$-simple groups are exactly the groups with the non-trivial amenable radical, as it turned out, the latter class forms not only a subclass, but a proper subclass of the former, as was shown by $\cite{Bo17}$, i.e. there are non-$C^*$-simple groups without non-trivial normal amenable subgroups.

In this text we would like to connect $C^*$-simplicity of groups with the behaviour of their Poisson boundary. 

Let us recall what is the Poisson boundary of a group and how the groups acts on it. Let $G$ be a countable group and $\nu$ a measure on $G$. We say that a measure is non-degenerate if its support generate $G$ as a semigroup.
Let $(X_i)_{i \in \N}$ be a $G$-valued i.i.d. process where each $X_i$ has distribution $\nu$. Let $(Z_i)_{i \in \N}$ be a process defined by $Z_i = X_1 \cdot \ldots \cdot X_i$. We endow the product $\Omega = \prod_{i = 1}^{\infty} G_i$, where $G_i$'s are copies of $G$, with the measure $\eta$ of path distribution of the $\nu$-random walk. In particular, $\pr_{G_i} \eta = \nu^{*i}$. There is an action of $G$ on $\Omega$: 
$$ \gamma \cdot (g_1, g_2, \ldots ) = (\gamma g_1, \gamma g_2, \ldots).$$
Let $\sA_i$ for $i \in \N$ denote the subalgebra of measurable subsets generated by the $G_i$ component of $\Omega$ (or equivalently, by $Z_i$). Let $\sA_{[i,\infty)}$ denote the join (the minimal subalgebra generated by) of $\sA_i, \sA_{i+1}, \ldots$.
The intersection $\sA_{\infty} = \bigcap_{i \in \N} \sA_{[i,\infty)}$ is the tail subalgebra of the random walk. Note that there is a unique up to isomorphism space $\partial_T(G, \nu)$ together with a natural map $\pr_{\partial_T} : \Omega \to \partial_T(G, \nu)$ such that $\sA_{\infty}$ is essentially the preimage of the natural Borel algebra on $\partial_T(G, \nu)$ under $\pr_{\partial}$. We note, that there is an induced quasi-invariant action of $G$ on $\partial(G, \nu)$ such that the map $\pr_{\partial} : \Omega \to \partial(G, \nu)$ is equivariant, if $\nu$ is non-degenerate. 
In general, the Poisson boundary is defined as the space of ergodic components under the shift-action, but in case of non-degenerate measure, it coincides with the tail boundary we defined earlier (see \cite{Ka92}).

For a countable set $X$ let us denote $\meas(X)$ the space of all probability measures on $X$ endowed with the total variation distance.
Our first main result is the following:

\begin{mytheor}\label{thm: main nonfree}
Let $G$ be a countable non-$C^*$-simple group. For a generic (resp. generic symmetric) probability measure $\nu$ on $G$ with respect to the total variation norm, the action of $G$ on the Poisson boundary is not essentilaly free.
\end{mytheor}

We will prove something stronger, for which we will need a new definition.

\begin{deff}
A finite subset $S$ of a group $G$ is called {\em amenably-visible} if for some amenable subgroup $H$ of $G$ we have that $S \cap H^\gamma \neq \varnothing$ for all $\gamma \in G$ (equivalently $S^\gamma \cap H \neq \varnothing$ for all $\gamma \in G$). 
\end{deff}
There are some easy examples. First, any subset that contains the group identity is trivially amenably-visible. Any subset that has a non-empty intersection with a normal amenable subgroup is amenably visible. 

Kennedy in \cite{Ke20} gave the following characterization of $C^*$-simplicity:

\begin{theor}\label{thm: inner criterion}
A group is not C*-simple iff it has a finite amenably-visible subset not containing the group identity.
\end{theor}

We will prove in Corollary \ref{cor: main non-free action} that for every finite amenably-visible subset $S$ of a countable group $G$, for a generic (resp. generic symmetric) measure $\nu$ the stabilizer of almost every point of the Poisson boundary $\partial(G, \nu)$ has non-empty intersection with $S$.

We might consider the following corollary as a strenghtening of the famous Kaimanovich-Vershik  and Rosenblatt construction \cite{KaVe83}, \cite{Ro81} of the measure on amenable groups with trivial Poisson boundary:
\begin{cor}
For a generic measure $\nu$ on the countable group $G$, the amenable radical acts trivially on the Poisson boundary $\partial(G, \nu)$.
\end{cor}

\begin{proof}
Every element of the amenable radical forms a one-element amenably-visible subset. An intersection of countably many Baire-residual sets is a Baire-residual set.
\end{proof}

This is not a coincidence, since  our construction could be considered, in a sense, an extension of that construction.

As we mentioned, Le Boudec proved in \cite{Bo17} that there are $C^*$-simple groups that have trivial amenable radical. Using the fact, that for a non-degenerate measure on a group, stabilizer of almost every point of the Poisson boundary is an amenable subgroup, we get:

\begin{cor}
Let $G$ be a non-$C^*$-simple group that has trivilal amenable radical. There is a non-degenerate measure on $G$ such that stabilizers of the points of the Poisson boundary are not essenitlly the same subgroup.
\end{cor}

Previouly, the only example of such phenomenon was due to Erschler an Kaimanovich \cite{ErKa24+}, they constructed a measure on the group of infinite permutations with finite support such that the action on the Poisson boundary is totally non-free in the terminology of A.M. Vershik (that is, the map that sends a point to its stabilizer is essentially bijective). 

For ICC groups there are always symmetric measures of full support such that the action of the Poisson boundary is essentailly free, this was proved in \cite{FHTF19}. It is not hard to observe that their construction could be modified to yield a dense collection of such measures. Thus, for an ICC non-$C^*$-simple group, for a generic probability measure on the group, the action on the Poisson boundary is not essentially free, but for a dense set of measures the action is essentially free.

Hartman and Kalantar gave a characterization of $C^*$-simplicity in terms of existence of so-called $C^*$-simple measures in \cite{HK23}. They also proved that the action on the Poisson boundary corresponding to a $C^*$-simple measure is essentially free, We augment their result slightly in  Section \ref{sec: essentially free} by noticing that a generic (resp. generic symmetric) probability measure on a $C^*$-simple countable group is a $C^*$-simple measure, see Corollary \ref{cor: cstar main cor}. Thus we get the following:

\begin{mytheor}
Let $G$ be a countable $C^*$-simple group. For a generic (resp. generic symmetric) probability measure $\nu$ on $G$, the action of $G$ on the Poisson boundary $\partial(G, \nu)$ is essentially free.
\end{mytheor}

Combining the two results with the classical Baire category theorem, we get a characterization of $C^*$-simplicity in terms of the freeness of the action of $G$ on the Poisson boundary $\partial(G, \nu)$ for a generic probability measure $\nu$ on $G$.

\begin{mytheor}
A countbale group $G$ is $C^*$-simple iff for a generic (symmetric) probability measure $\nu$ on it, the action on the Poisson boundary $\partial(G, \nu)$ is essentailly free. A countbale group $G$ is not $C^*$-simple iff for a generic (symmetric) probability measure $\nu$ on it, the action on the Poisson boundary $\partial(G, \nu)$ is not essentailly free. 
\end{mytheor}

{\em Acknowledgements.} I would like to thank Vadim Kaimanovich and Romain Tessera for discussions. I'm grateful to Anna Erschler for comments and suggestions.
I finised writing this text during my visit to RIMS Kyoto hosted by Narutaka Ozawa. This text is an expansion upon my preprint ``Non-C*-simple groups admit non-free actions on their Poisson boundaries'' where only one side was provided.

\section{A criterion for non-triviality of stabilizers}\label{sec: stab nontriv}


The next lemma provides us with an important criterion of non-freeness of the action on the Poisson boundary.

\begin{lem}\label{lem: asymptotic disjointness}
Let $\nu$ be a measure of full support on a group $G$. Let $S$ be a finite subset of $G$ and assume that there is a positive-measure subset of elements of the Poisson boundary $\partial(G, \nu)$ those $G$-stabilizers are disjoint with $S$. For every $\varepsilon > 0$ there is probability measure $\mu$ on $G$ such that 
\begin{equation*}
\liminf_{n \to \infty}\lvert t \cdot \mu * \nu^{*n} - \mu * \nu^{*n} \rvert \geq 2 - \varepsilon
\end{equation*}
for all $t \in S$.
\end{lem} 
\begin{proof}
The argument is similar in flavor to the proof of the 0-2 law in \cite{Ka92}.

We first note that there is a positive-measure subset $Q'$ of $\partial(G, \nu)$ such that $tQ'$ does not intersect $Q'$ for all $t \in S$. This corresponds to a subset $Q$ of $\Omega$. Let $\eta_Q$ be the normalized restriction of $\eta$ to $Q$. Note that measures $\eta_Q$ and $t \cdot \eta_Q$ are supported on disjoint sets $Q$ and $tQ$, so $\lvert t \cdot \eta_Q  - \eta_Q\rvert  = 2$. Now let $f$ be the Radon-Nidodym derrivative of $\eta_Q$ with respect to $\eta$: $\eta_Q = f \eta$. In the same way, let $f_t$ be the derrivative of $t \cdot \eta_Q$ with respect to $\eta$ for all $t \in G$:
$t \cdot \eta_Q = f_t \eta$. We observe that
\begin{equation}\label{eq: asymptotic disjointness in the limit}
\lvert f - f_t\rvert_{L^1} = 2,
\end{equation}
for all $t \in S$. Using the martingale convergence theorem, we get that for big enough $i$ we have $\vert \expect(f \vert \sA_{[0,i]}) - f \rvert_{L^1} < \varepsilon/2$. 
Denote $f^i =  \expect(f \vert \sA_{[0,i]})$. 
and also $f^i_t$ the Radon-Nikodym derrivative of $t \cdot f^i\eta$ over $\eta$: $f^i_t \eta = t \cdot f^i \eta$. 

We claim that 
\begin{equation}\label{eq: asymptotic disjointness almost}
\lvert \expect(f^i_t \vert \sA_{[j,\infty]}) - \expect(f^i \vert \sA_{[j,\infty]})\rvert_{L^1} >  2 - \varepsilon,
\end{equation}
for $t \in S$.
To get this we observe that 
$$
\lvert f^i - f \rvert_{L^1} < \varepsilon/2.
$$
and 
$$
\lvert  f^i_t - f_t \rvert_{L^1} = \lvert t \cdot f^i \eta - t \cdot \eta_Q\rvert = \lvert f^i \eta - \eta_Q\rvert = \lvert f^i - f \rvert_{L^1} < \varepsilon/2.
$$
So 
\begin{multline*}
\lvert \expect(f^i_t \vert \sA_{[j,\infty]}) - \expect(f^i \vert \sA_{[j,\infty]})\rvert_{L^1} \\= \lvert \expect(f^i_t  - f^i \vert \sA_{[j,\infty]})\rvert_{L^1} > \lvert\expect(f_t  - f \vert \sA_{[j,\infty]})\rvert_{L^1} - \varepsilon.
\end{multline*}
Now, recalling that both $f_t$ and $f$ are $\sA_{[j,\infty]}$-measurable (in fact, $\sA_{\infty}$-measurable), and using \eqref{eq: asymptotic disjointness in the limit}, we get that the last term is equal to $2 - \varepsilon$, finishing the proof of \eqref{eq: asymptotic disjointness almost}.

Let $\mu = \pr_{G_i}(f^i \eta)$. 

We now claim that 
\begin{equation}
\lvert \expect(f^i_t \vert \sA_{[j,\infty]}) - \expect(f^i \vert \sA_{[j,\infty]})\rvert_{L^1} = \lvert \mu * \nu^{*(j-i)} - t \cdot \mu * \nu^{*(j-i)}\rvert,
\end{equation}
for all $t \in G$ and $j > i$.
this will finish the proof due to \eqref{eq: asymptotic disjointness almost}.

Observe that due to the Markov property, $f^i$ is $\sA_i$-measurable. Also, $f^i_t$ is $\sA_{[1,i]}$-measurable.
Note that $\pr_{G_j}(f^i\eta) = \mu * \nu^{*(j-i)}$ and $\pr_{G_j}(f_t^i\eta) = t \cdot \mu * \nu^{*(j-i)}$. Now it is enough to observe that 
$$\lvert \expect(f^i_t \vert \sA_{[j,\infty]}) - \expect(f^i \vert \sA_{[j,\infty]})\rvert_{L^1} = \lvert \pr_j(f^i\eta - f^i_t\eta)\rvert$$


\end{proof}

\section{A decomposition lemma}

This lemma will be instrumental for our construction.

\begin{lem}\label{lem: second decomp}
Let $(\upsilon_i)_{i=0}^{\infty}$ be a probability vector such that for the integer-valued i.i.d. process $(K_i)$ with distribution $(\upsilon_i)$ we have $\limsup_{i \to \infty}(K_i - i) > 0$. Let $(\zeta'_i), (\zeta''_i)$ be some probability measures on the group $G$ and $(\beta_i)$ be a sequence of numbers $0 \leq b_i \leq 1$ such that $\liminf_i{\beta_i} > 0$.
Denote $A_n = \bigcup_{i \leq n}\big(\supp \zeta'_i \cup \supp \zeta''_i\big)$, and $$\nu = \sum_{i=0}^{\infty}\upsilon_i(\beta_i\cdot \zeta'_i + (1 - \beta_i) \cdot \zeta''_i),$$ clearly; $\nu$ is  probability measure. We claim that for every $\varepsilon > 0$ and $M > 0$ there is $N$ such that for all $n > N$ there is a decomposition
\begin{equation}\label{eq: second decomp}
\nu^{*n} = \eta_{n, \varepsilon} + \sum_{i \geq M}\sum_{q' \in A^{<i-1}_{i-1}} \sum_{q'' \in G} \alpha_{i,q',q''} \cdot q' * \zeta'_i * q''
\end{equation}
where $\eta_{n,\varepsilon}$ is a positive measure with $\lvert \eta_{n, \varepsilon}\rvert < \varepsilon$, $\alpha_{i,q',q''} \geq 0$, and $\lvert \eta_{n,\varepsilon}\rvert + \sum_{i,q',q''}\alpha_{i,q',q''} = 1$.
\end{lem}

\begin{proof}
We consider the i.i.d. process $(K_i)$.
Let $(Y_i)$ be a process with values in the set $\{1,2\}$, where each $Y_i$ is coupled to corresponding $K_i$ in such a way that $Y_i = 1$ with probability $\beta_{K_i}$ and $Y_i = 2$ with probability $1 - \beta_{K_i}$. Next, let $X_i$ for each $i$ be distibuted according to $\zeta'_{K_i}$ if $Y_i = 1$ and according to $\zeta''_{K_i}$ if $Y_i = 2$.
This way we get that $(X_i)$ is an i.i.d. process with each $(X_i)$ having the distribution $\nu$.
We define the stopping time $T$ to be the first $t$ such that $K_t > t$, $K_t \geq M$ and $Y_t = 1$. Note that this is indeed a stopping time and it is finite almost surely. Take $N$ such that $T > N$ with probability smaller than $\varepsilon$. Take initial segment $k_1, \ldots, k_n$ with $n > N$ of a trjectory of the porcees $(K_i)$ such that the stopping time $T = t \leq n$ (the probaility of the latter event is bigger than $1 - \varepsilon$). Take the relative realization of $X_1, \ldots , X_{t-1}$, these are some elements $x_1, \ldots, x_{t-1} \in A_{k_{t-1}}$, and the relative realization of $X_{t+1}, \ldots , X_n$, these are elements $x_{t+1}, \ldots,x_n \in G$. Now observe that the relative ditribution of $X_t$ is $\mu'_{k_t}$ (since $Y_t = 1$ by the definition of the stopping time). This means that for the product $X_1 \cdot \ldots \cdot X_n$ we have the distribution $q' * \mu_{i} * q''$, where $i = k_t$, $q' \in A_{i-1}^{<t} \subset  A_{i-1}^{<i-1}$ (since $i = k_t > t$ by definition of the stopping time). It remains to note that $\eta_{n, \varepsilon}$ is the leftower measure correponding the event $T > n$.
\end{proof}

\section{Generic measures have non-trivial stabilizers}

Let $G$ be a countable non C*-simple group. Let $(\sigma_i)$ be a dense in the total variation norm sequence of probability mesures on $G$ each of which has finite support. Let $S$ be a finite non-empty amenably-visible set.
We denote $\clNF_S$ the set of all such probability measures $\nu$ on $G$ that for every $\sigma_i$ there is $n$ and $s \in S$ such that 
$$
\lvert s \cdot \sigma_i \cdot \nu^{*n} - \sigma_i \cdot \nu^{*n}\rvert < 2 \Big(1 - \frac{1}{2\lvert S \rvert}\Big).
$$

Note that for a fixed $\sigma_i$, this requirements defines an open set of measures, hence we get:

\begin{lem}\label{lem: NF is Gdelta}
Class $\clNF_S$ forms a $G_{\delta}$ subset of $\meas(G)$, as well as of the subspace of symmetric measures.
\end{lem}

We also note that 

\begin{lem}
If a probability measure $\nu$ belongs to $\clNF_S$ then 
$$
\lim_{i \to \infty}\lvert s \cdot \sigma \cdot \nu^{*n} - \sigma \cdot \nu^{*n}\rvert \leq 2 \Big(1 - \frac{1}{2\lvert S \rvert}\Big),
$$
for every probability measure $\sigma$ on $G$, and the limit exists.
\end{lem}
\begin{proof}
The existence of the limit follows from the monotonicity of the quantity and the latter, in turn, is implied by the fact that the convolution with a probability measure is a Markov operator. For $\sigma = \sigma_i$ the statement is now trivial and for other probability measures we deduce it from continuity.
\end{proof}

As a trivial consequence of Lemma \ref{lem: asymptotic disjointness}, we get the following:

\begin{lem}\label{lem: NF stabilizers}
If a measure $\nu$ on the group $G$ is non-degenerate and belongs to $\clNF_{S}$ for some finite non-empty subset $S$ of $G$, then the stabilizer of almost every point of the Poisson boundary $\partial(G, \nu)$ has non-empty intersection with $S$.
\end{lem}

We aim to prove that this class is in fact dense among the probability measures on $G$.

We will prove the following: 

\begin{lem}\label{lem: main non-free construction}
Let $\nu'$ be any probability measure on the group $G$. Let $S$ be a finite anemably-visible sibset of $G$. There is a measure $\nu''$ on $G$ such that for every $0 < \theta  < 1$ we have that $\theta \nu'' + (1 - \theta) \nu' \in \clNF_S$.
\end{lem}

Combining this lemma with lemmata \ref{lem: NF is Gdelta} and \ref{lem: NF stabilizers}, we ge the following;
\begin{cor}\label{cor: main non-free action}
For any finite non-empty amenably-visible subset $S$ of a countable group $G$, for a generic probability measure $\nu$ (resp. generic symmetric probability measure) on the group, the stabilizer of almost every point of the Poisson boundary $\partial(G,\nu)$ has a nonempty intersection with $S$.
\end{cor}

We first give a construction and then returrn to the proof of the lemma.

Let us make the following easy observation without a proof:

\begin{lem}\label{lem: finitary decomposition}
Let $\nu'$ be a probability measure on a countable set and let $(\upsilon_i)_{i = 1}^{\infty}$ be a proibability vector ($\upsilon_i \geq 0$ and $\sum_i \upsilon_i = 1$) with infinitely many positive elements. There is a sequence of finitely supported probability measures $(\nu'_i$) such that 
$$
\nu' = \sum_i \upsilon_i \nu'_i. 
$$
\end{lem}

Since $S$ is an amenably-visible subset, let us fix an amenable subgroup $H$ of $G$ such that $S^\gamma \cap H \neq \varnothing$ for all $\gamma \in G$.

We say that a finite subset $F$ of the group $G$ is $(R, \epsilon)$-invariant for a finite subset $R$ of $G$ and $\epsilon > 0$ if $\lvert RF \setminus F \rvert <  \epsilon \lvert F\rvert$.
We also introduce the following notation $S^A = \lbrace s^a = a^{-1} s a \vert s \in S, \, a\in A\rbrace$.

We fix an enumeration $(c_i)_{i=1}^{\infty}$ of all elements of $G$.
Let us iteratively construct sequences of finite subsets $(A_i)_{i=0}^{\infty}$ and $(F_i)_{i=1}^{\infty}$ of $G$. We set $A_0 = \varnothing$. Now for each $i \geq 1$ we define
\begin{enumerate}
\item Let $F_i$ be any finite symmetric $(S^{A_{i-1}^{i}} \cap H, 1/i)$-invariant subset of $H$;
\item Let $A_i = A_{i-1} \cup \lbrace c_i , c^{-1}_i\rbrace \cup \supp \nu'_i \cup (\supp \nu'_i)^{-1}$. 
\end{enumerate}

We define the probability measure 
\begin{equation*}
\nu'' = \sum_{i = 1}^{\infty} \frac{\upsilon_i}{3}(\delta_{c_i} + \delta_{c_i^{-1}} + \lambda_{F_i}),
\end{equation*}
where $\lambda_{F_i}$ is the uniform probability measure on the set $F_i$.
It is easy to observe that $\nu''$ is a symmetric measure of full support on $G$.

Let us carry on with the proof.

\begin{proof}[Proof of Lemma \ref{lem: main non-free construction}]
Consider $\nu = \theta \nu''  + (1 - \theta)\nu'$ for some $0 < \theta < 1$.
We will prove that for every finitely-supported probability measure $\sigma$ on $G$ we have 
$$
\lvert s \cdot \sigma \cdot \nu^{*n} - \sigma \cdot \nu^{i}\rvert < 2 \Big(1 - \frac{1}{2\lvert S \rvert}\Big),
$$
for all big enough $n$.

Take $\varepsilon > 0$. We apply the decomposition from Lemma \ref{lem: second decomp} to $\nu$, using
$\zeta'_i = \lambda_{F_i}$, $$\zeta''_i = \frac{1}{1 - \theta/3}\Big(\frac{\theta}{3}(\delta_{c_i} + \delta_{c_i^{-1}}) + (1 - \theta)\nu'_i \Big),$$
$\beta_i = \theta/3$, and $M$ such that $1/M < \varepsilon$ and $\supp \sigma \subset A_{M-1}$.
This gives us, that there is a certain $N$ such that for all $n > N$ we have

\begin{equation}\label{eq: second decomp}
\nu^{*n} = \eta + \sum_{i \geq M}\sum_{q' \in A^{<i-1}_{i-1}} \sum_{q'' \in G} \alpha_{i,q',q''} \cdot q' * \lambda_{F_i} * q'',
\end{equation}
where $\eta$ is a positive measure with $\lvert \eta\rvert < \varepsilon$, $\alpha_{i,q',q''} \geq 0$, and $\lvert \eta\rvert + \sum_{i,q',q''}\alpha_{i,q',q''} = 1$.

For $\sigma * \nu^{*n}$ we get a decomposition 

\begin{equation}\label{eq: convoluted decomp}
\sigma * \nu^{*n} = \eta' + \sum_{i \geq M}\sum_{q' \in A^{<i}_{i-1}} \sum_{q'' \in G} \alpha'_{i,q',q''} \cdot q' * \lambda_{F_i} * q'',
\end{equation}
where $\eta' = \sigma * \eta$ is a positive measure with $\lvert \eta'\rvert < \varepsilon$, $\alpha_{i,q',q''} \geq 0$, and $\lvert \eta'\rvert + \sum_{i,q}\alpha'_{i,q',q''} = 1$. Now we note that for each $q'$ there is an $s \in S$ such that $s^{q'} \in H$. Thus, we may choose an $s \in S$ such that 
$$
\sum_{i,q',q''}\alpha'_{i,q',q''} \geq \frac{1 - \varepsilon}{\lvert S\rvert},
$$
where the sum is taken over $i,q',q''$ such that $i \geq M$, $q' \in A^{<i}_{i-1}$, $s^{q'} \in H$, $q'' \in G$. Now, noting that 

$$ s * q' * \lambda_{F_i} * q'' = q' * s^{q'} * \lambda_{F_i} * q'',$$
and using the $(S^{A^i_{i-1}} \cap H, 1/M)$-invariance of $F_i$, we conclude that

$$
\lvert s \cdot \sigma \cdot \nu^{*n} - \sigma \cdot \nu^{i}\rvert < 2 \Big(1 - \frac{1}{\lvert S \rvert}\Big) + 4 \varepsilon,
$$ 
for all $n > N$.

\end{proof}

\section{Genericlly essentially free action on the Poisson boundaries of $C^*$-simple groups}\label{sec: essentially free}
Denote $\Span S$ the linear span of a set $S$. 

Let us denote $\fa_g$ the element of $C_r(G)$ corresponding to $g \in G$.
We note that there is a natural right action of the group $G$ on $C_r(G)$ by conjugation: $a^g = \fa_g^{-1} a \fa_g$, for $g \in G$ and $a \in C_r(G)$. This is an action by $C^*$-algebra automorphisms. For a state $\phi \in \bS(C_r(G))$ and a group element $g \in G$ we may define $(g \cdot \phi)(a) = \phi(a^g)$ for each $a \in C_r(G)$. For a finite measure $\nu$, $a \in C_r(G)$ and $\phi \in \bS(C_r(G))$ we may define $(\nu \cdot \phi) = \sum_{g \in G} \nu(\lbrace g\rbrace)(g \cdot \phi)$ and $a^{\nu} = \sum_{g \in G} a^g$. Note that $(\nu \cdot \varphi)(a) = \varphi(a^{\nu})$, If $\nu$ is a probability measure then the map $a \mapsto a^\mu$ is a linear operator of norm $1$: the norm is trivially bounded by $1$ and $\fa_e^\nu = \fa_e$.

Denote $\tau : C_r(G) \to \C$ the natural trace on $C_r(G)$ defined by $\tau(a) = \langle a v , a\rangle$ where $v$ is the unit vector from $l^2(G)$ corresponding to the group identity. We say that a $\varphi$ state is $\nu$ - stationary if $\varphi = \nu \cdot \varphi$. We say (following \cite{HK23})that a probability measure $\nu$ is $C^*$ - simple if $\tau$ is the only $\nu$-stationary state. 
Hartman and Kalantar gave in \cite{HK23} the followitng characterization for $C^*$-simplicity of a group:
\begin{theor}
A group is $C^*$-simple iff it has a $C^*$-simple probability measure. 
\end{theor}

As a part of this theorem, authors presented a construction of a $C^*$-simple measure on a $C^*$-simple group. We aim to show that such measures are generic in the spaces of probability measures and in the space of symmetric probability measures endowed with the total variation norm. 

Importance of $C^*$-simple measures for us stems from the following theorem that again can be found in \cite{HK23}:

\begin{theor}
Let $\nu$ be a non-degenerate $C^*$-simple measure on a countable group $G$. The natural action $G \curvearrowright \partial(G, \nu)$ is essentially free.
\end{theor}
The main idea there is that the stabilizer of almost every point is an amenable subgroup, and for each amenable subgroup there is a natural state on $C_r(G)$ that has the value $1$ on the group algebra element corresponding to the sungroup elements, and the value $0$ for algebra elements corresponding to other group elements . Integrating these states, we get a $\nu$-stationary state that is not equal to the natural trace.

Let us note that $\ker \tau = \overline{\Span \lbrace \fa_g \vert g \in G \setminus \lbrace e\rbrace \rbrace}$. 
Note that \[
{C_r(G) = \ker \tau \oplus \Span \lbrace \fa_e\rbrace},
\]
and that these subspaces are invariant for the operator $a \mapsto a^{\nu}$.
We say that a probability measure $\nu$ on the group $G$ is a {\em Hartman-Kalantar  measure, or HK-measure} for $G$ if $\lim_{n \to \infty}\lVert a^{\nu^{*n}}\rVert = 0$ for all $a \in \ker \tau$.
We get the following:

\begin{lem}
Any HK probability measure is $C^*$-simple.
\end{lem}
\begin{proof}
Let $\nu$ be an HK probability measure and let $a \in C_r(G)$. Consider decomposition $a = x + y$, where $x \in \ker \tau$ and $y \in \Span \lbrace \fa_e \rbrace$. 
Let $\varphi$ be a $\nu$-stationary state. Take any $a \in C_r(G)$ and compute:
\begin{multline*}
\varphi(a) = \lim_{n \to \infty} (\nu^{*n} \cdot \varphi)(a) = \lim_{n \to \infty} \varphi(a ^ {\mu^{*n}}) =\\  \lim_{n \to \infty} \varphi(x ^ {\mu^{*n}}) + \lim_{n \to \infty} \varphi(y ^ {\mu^{*n}}) = \varphi(y) = \tau(y) = \tau(a).
\end{multline*}
\end{proof}


We will show that the set of HK probability measures is residual in the space of all probabiltiy measures on a $C^*$-simple group endowed with the topology of total variation. Let us first establish that HK probability measures form  a $G_\delta$ set. For this we will need the following lemma.

\begin{lem}
Let $G$ be a countable group, $\nu$ be a probability measure on $G$ and $(d_i)$ be a dense sequence in the unit ball of $\ker \tau$.
The following are equivalent:
\begin{enumerate}
\item $\nu$ is an HK measure;
\item for any $a$ in the unit ball of $\ker \tau$ there is a natural $n$ such that $\lVert a^{\nu^{*n}} \rVert < 1/2$.
\item for any natural $i$ there is $n$ such that $\lVert d_i^{\nu^{*n}} \rVert < 1/2$
\end{enumerate}
\end{lem}
\begin{proof}
Implications $(1) \Rightarrow (2) \Rightarrow (3)$ are trivial. 

Let us prove that  $(2) \Rightarrow (1)$ we note that $(2)$ implies that for any $a$ from $\ker \tau$ there is a natural $n_1$ such that $\lVert a^{\nu^{*n_1}} \rVert \leq 1/2 \lVert a\rVert$. Now, applying this same observation to $a^{\nu^{*n_1}}$ we get $n_2$ such that $\lVert a^{\nu^{*(n_1 + n_2)}} \rVert \leq 1/4 \lVert a \rVert$. Repeating this, we get that for every $m$ there is $r$ such that  $\lVert a^{\nu^{*r}} \rVert \leq 1/2^m \lVert a\rVert$.

Let us prove that $(3) \Rightarrow (2)$. Recall that the norm of the operator $a \mapsto a^{\nu}$ is $1$. For any $a$ in the unit ball of $\ker \tau$ there is $d_i$ at the distance smaller that $1/6$. Applying $(3)$ to this $d_i$ we get that there is $n$ such that $\lVert a^{\nu^{*n}}\rVert < 2/3$. This means that for any $a \in \ker \tau$ there is $n$ such that $\lVert a^{\nu^{*n}}\rVert \leq 2/3 \lVert a\rVert$. Using an argument as in the previous paragraph, we conclude that for any $a$ from the unit ball of $\ker \tau$ there is $n$ such that $\lVert a^{\nu^{*n}}\rVert < 1/2$.
\end{proof}

\begin{cor}
The set of all HK measures is 
\[
\bigcap_{i \in \N}\bigcup_{n \in \N} \lbrace \nu \in \meas(G) \text{ such that } \lVert d_i^{\nu^{*n}}\rVert < 1/2\rbrace.
\]
\end{cor}

\begin{cor}
The set of all HK measures on a countable group is an intersection of contably-many open sets.
\end{cor}
\begin{proof}
In view of the previous corollary, it is enough to observe that for any $a \in C_r(G)$ and an integer $n$, the map $\nu \mapsto a^{\nu^{*n}}$ is continuous from $\meas(G)$ to $C_r(G)$.
\end{proof}

Next we would like to prove that the set of HK measures is dense. For this we will prove later in this section the following:

\begin{lem}\label{lem: HK density}
Let $G$ be a $C^*$-simple group. Let $\mu_0$ be a probability measure of finite support on $G$ such that $\lvert \nu'\rvert < 1$. 
For all $0 < \delta < 1$ there is a symmetric probagbility measure $\nu'$ on $G$ such that  the measure $(1 - \delta)\mu_0 + \delta\nu'$ is an HK probability measure.
\end{lem}

\begin{rem}
Using the approach similar to that used in Section \ref{sec: stab nontriv}, we can get that for every probability measure $\mu_0$ on a countable $C^*$-simple group $G$ there is a symmetric measure $\nu'$ of full support such that for every $0 < \theta \leq 1$, the convex combination $\theta \nu' + (1 - \theta)\mu_0$ is an HK-measure.  
\end{rem}

\begin{cor}\label{cor: cstar main cor}
For a countable $C^*$-simple group, the set of all HK-probability  measures is a dense with respect to the total variation topology $G_{\delta}$ set in both the set of probability measures and the set of symmetric probability measures. Thus $C^*$-simple measures are generic, and the action on the Poisson boundary is essentially free for a generic probability (resp. generic symmetric probability) measure on a countable $C^*$-simple group. 
\end{cor}

A group $G$ is said to satisfy the Powers property if for every $a \in \ker \tau$ and each $\varepsilon > 0$ there is a finitely-supported probability measure $\nu$ on $G$ such that $\lVert a^\nu\rVert < \varepsilon$. In his pioneering work on $C^*$-simplicity \cite{Po75}, Powers gave the following criterion:

\begin{theor}
A groups that satisfies the Powers property is $C^*$-simple.
\end{theor}

This property turned out to be a characterization of $C^*$-simplicity, as was shown by Haagerup \cite{Ha16} and Kennedy \cite{Ke20}:
\begin{theor}
A groups is $C^*$-simple if and only if it satisfies the Powers property.
\end{theor} 

We will need the following simple but convenient strengthening of the Powers property.

\begin{lem}\label{lem: powers strong}
If $G$ is a group with the Powers property, then for every finite $S \subset \ker \tau$ and for every $\varepsilon > 0$ there is a finitely-supported symmetric probability measure $\nu$ such that $\lVert a^{\nu}\rVert < \varepsilon$ for each $a \in S$.
\end{lem}
\begin{proof}
First we note that it is enough to construct a possibly asymmetric measure $\nu$ and after that take $\nu * \nu^{-1}$ as the final result.

We will prove the statement (without requiring the measure to be symmetric) by induction. Let us fix $\varepsilon > 0$.
For an empty set $S \subset \ker \tau$ we can take any measure. Assume we constructed the measure $\nu$ for a  finite set $S \subset \ker \tau$, let us construct a measure for $S' = S \cup \lbrace a\rbrace$, for some $a \in \ker \tau$. By definition of the Powers property, there is a measure $\nu_1$ such that $\lVert ({a}^{\nu})^{\nu_1}\rVert < \varepsilon$, since $a^{\nu} \in \ker \tau$. Now it is easy to see that the measure $\nu' = \nu * \nu_1$ is such that 
$\lVert s^{\nu}\rVert < \varepsilon$ for every $s \in S'$.
\end{proof}

Now, we would like to describe the construction, but first let us introduce a suitable variant of the decomposintion lemma.

\begin{lem}\label{lem: second decomp}
Let $(\upsilon_i)_{i=0}^{\infty}$ be a probability vector such that for the integer-valued i.i.d. process $(K_i)$ with distribution $(\upsilon_i)$ we have $\limsup_{i \to \infty}(K_i - i) = \infty$. Let $\mu_i$ be some probability measures on a group $G$. 
Denote $A_n = \bigcup_{i \leq n}\supp \mu_i$, and $\nu = \sum_{i=0}^{\infty}\upsilon_i\mu_i$ (clearly, this is a probability measure). We claim that for every $\varepsilon > 0$ and $M > 0$ there is $N$ such that for all $n > N$ there is a decomposition
\begin{equation}\label{eq: second decomp}
\nu^{*n} = \eta_n + \sum_{i \geq M}\sum_{q' \in A^{<i}_{i-1}} \sum_{q'' \in G} \alpha_{i,q',q''} \cdot q' * \mu_i * q'',
\end{equation}
where $\eta_{n,\varepsilon}$ is a positive measure with $\lvert \eta_{n, \varepsilon}\rvert < \varepsilon$, $\alpha_{i,q',q''} \geq 0$, and $\lvert \eta_{n,\varepsilon}\rvert + \sum_{i,q}\alpha_{i,q} = 1$.
\end{lem}

\begin{proof}
We consider the i.i.d. process $(K_i)$ and let $(X_i)$ be a process where each $X_i$ is coupled to $K_i$ in the following way: the distribution of $X_i$ relative to $K_i$ is $\mu_{K_i}$. This implies that the whole distribution of $X_i$ is precisely $\nu$. We define the stopping time $T$ to be the first $t$ such that $K_t > t$ and $K_t \geq M$. Note that this is indeed a stopping time and it is finite almost surely. Take $N$ such that $T > N$ with probability smaller than $\varepsilon$. Take initial segment $k_1, \ldots, k_n$ with $n > N$ of a trjectory of the process $(K_i)$ such that the stopping time $T = t \leq n$ (the probaility of the latter event is bigger than $1 - \varepsilon$). Take the relative realization of $X_1, \ldots , X_{t-1}$, these are some elements $x_1, \ldots, x_{t-1} \in A_{k_{t-1}}$, and the realtive realization of $x_{t+1}, \ldots, x_n \in G$. Now observe that the relative ditribution of $X_t$ is $\mu_{k_t}$. This means that for the product $X_1 \cdot \ldots \cdot X_n$ we have the distribution $q' * \mu_{i} * q''$, where $i = k_t$, $q' \in A_{i-1}^{< t-1} \subset  A_{i-1}^{< i-1}$ (since $i = k_t > t$ by definition of the stopping time). It remains to note that $\eta$ is the leftover measure correponding the event $T > n$.
\end{proof}

\begin{proof}[Proof of Lemma \ref{lem: HK density}]
First let us consider a dense in the unit ball off $\ker \tau$ sequence $d_i$. 
Also fix a sequence $(c_i)_{i \geq 1}$ that enumerates $G$.
Take a $0 < \delta < 1$. Let us constuct a probability vector $(\upsilon_i)$ such that $\upsilon_0 = 1 - \delta$ and for the i.i.d. process $(K_i)$ with distribution $(\upsilon_i)$ we will have $\limsup_{i \to \infty} K_i - i = \infty$. For that it is enough to set $\upsilon_i \sim i^{-\alpha}$ for some $1 < \alpha < 2$. Now we will constuct iteratively the finite subsets $A_i$ of the group $G$ and probability measures $\mu_i$. For $i > 0$ we set $\mu_i$ to be a symmetric probability measures such that for all $j < i$ and all $g \in A^{<i}_{i-1}$ we have $\lVert d_j^{g * \mu_i} \rVert < 1/4$, and also $c_i, c_i^{-1} \in \supp \mu_i$; such measure exists by Lemma \ref{lem: powers strong} (we might need to add a little weight to $c_i$ and $c_i^{-1}$). We then set $A_i = A_{i-1} \cup \supp \mu_i$. Note that we tailored the construction, so that we can apply the decomposition from Lemma \ref{lem: second decomp} to the weighted sum $\nu = \sum_{i = 0}^{\infty}$. Thus for every $M > 0$ and for $\varepsilon = 1/4$ there is $N > 0$ such that for all $n > N$ we will have the decomposition \eqref{eq: second decomp}. 
Now observe that 
\begin{equation*}\label{eq: basic estimate in HK density proof}
\lVert d_M^{q' * \mu_i * q''} \rVert \leq \Big\lVert {{(d_M^{q'})} ^ {\mu_i}} \Big\rVert \leq 1/4.
\end{equation*}
for all $q'' \in G$, $i \geq M$ and $q' \in A_{i-1}^{<i}$. From this and the decomposition \ref{eq: second decomp} we get that 
\begin{multline*}
\lVert d_M^{\nu^{*n}} \rVert  \leq \lVert {d_M}^{\eta_i} \rVert + \sum_{i \geq M}\sum_{q' \in A^{<i}_{i-1}} \sum_{q'' \in G} \alpha_{i,q',q''} \cdot \lVert d_M^{q' * \mu_i * q''}\rVert < \lvert \eta_n \rvert + 1/4 < 1/2,
\end{multline*}
we used that $\alpha_{i,q',q''}$ are non-negative numbers with the sum bounded by $1$ and that $\lvert \eta_n\rvert < \varepsilon = 1/4$.
We get that $$\nu = (1 - \delta)\mu_0 + \delta \cdot \frac{\sum_{i = 1}^{\infty}\upsilon_i \mu_i}{\delta}.$$ Now $$\nu' = \frac{\sum_{i = 1}^{\infty}\upsilon_i \mu_i}{\delta}$$ is the requested probability measure.
\end{proof}



\begin{thebibliography}{100000}


\bibitem[Bo17]{Bo17} A. Le Boudec \textit{C*-simplicity and the amenable radical}, Inventiones mathematicae 209.1 (2017): 159--174.

\bibitem[BKKO17]{BKKO17} E. Breuillard, M. Kalantar, M. Kennedy, and N. Ozawa, \textit{C*-simplicity and the unique trace property for discrete groups}, Publications mathématiques de l'IHÉS 126, no. 1 (2017): 35--71.

\bibitem[ErKa19]{ErKa19} A. Erschler and V. Kaimanovich, \textit{Arboreal structures on groups and the associated boundaries}, arXiv preprint arXiv:1903.02095 (2019).
\bibitem[ErKa24+]{ErKa24+} A. Erschler and V. Kaimanovich, \textit{Totally non-free Poisson boundary}, in preparation.

\bibitem[FHTF19]{FHTF19} J. Frisch,Y Hartman, O. Tamuz, and P. V. Ferdowsi. \textit{Choquet-Deny groups and the infinite conjugacy class property}, Annals of Mathematics 190, no. 1 (2019): 307--320.

\bibitem[Ha16]{Ha16} Uffe Haagerup  \textit{A new look at $C^*$-simplicity and the unique trace property of a group} Operator Algebras and Applications: The Abel Symposium 2015. Springer International Publishing, 2016.

\bibitem[HK23]{HK23} Hartman, Yair, Mehrdad Kalantar, with appendix by Uri Bader, \textit{Stationary C*-dynamical systems}, Journal of the European Mathematical Society (EMS Publishing) 25, no. 5 (2023).
\bibitem[Ka92]{Ka92}V.A.  Kaimanovich, \textit{Measure-theoretic boundaries of Markov chains, 0–2 laws and entropy}, Harmonic analysis and discrete potential theory. Springer, Boston, MA, 1992. 145--180.
659--692, 
\bibitem[KaVe83]{KaVe83} V.A. Kaimanovich, and A. M. Vershik, \textit{Random walks on discrete groups: boundary and entropy}, The annals of probability (1983): 457--490.

\bibitem[KaKe17]{KaKe17} M. Kalantar and M. Kennedy, \textit{Boundaries of reduced C*-algebras of discrete groups}, Journal f\"ur die reine und angewandte Mathematik (Crelles Journal) 2017, no. 727 (2017): 247--267.

\bibitem[Ke20]{Ke20}M. Kennedy, \textit{An intrinsic characterization of C*-simplicity}, Ann. Scient. \'Ec. Norm. Sup., (4)53, 2020, 1105--1119,

\bibitem[Po75]{Po75} R. T. Powers, \textit{Simplicity of the $C^*$-algebra associated with the free group on two generators}, (1975): 151--156.
\bibitem[Ro81]{Ro81} J. Rosenblatt, \textit{Ergodic and mixing random walks on locally compact groups}, Mathematische Annalen 257 (1981), no. 1, 31--42.
\end{thebibliography}
\end{document}